\documentclass[11pt,dvipdfmx]{amsart}
\usepackage{amsmath,amssymb,amsthm}
\usepackage{ascmac}
\usepackage{color}
\usepackage{hyperref}
\usepackage[english]{babel} 
\usepackage{tikz}
\usepackage{enumitem}

\theoremstyle{definition}
\newtheorem{theorem}{Theorem}[section]
\newtheorem{remark}[theorem]{Remark}
\newtheorem{lemma}[theorem]{Lemma}

\newtheorem{corollary}[theorem]{Corollary}
\newtheorem{definition}[theorem]{Definition}

\newtheorem*{proof of claim}{Proof of Claim}

\newcounter{claimcounter}
\newtheorem*{claim}{Claim \thetheorem.\theclaimcounter}

\DeclareMathOperator{\ATRo}{\mathsf{ATR}_0}
\newcommand{\ACAo}{\mathsf{ACA}_0}
\newcommand{\RCAo}{\mathsf{RCA}_0}
\newcommand{\CAo}{\mathsf{CA}_0}

\newcommand{\SDC}{\mathsf{SDC}}

\newcommand{\ACA}{\mathsf{ACA}}

\DeclareMathOperator{\HJ}{HJ}

\newcommand{\N}{\mathbb{N}}

\renewcommand{\L}{\mathcal{L}}

\newcommand{\M}{\mathcal{M}}
\renewcommand{\phi}{\varphi}
\newcommand{\RFN}{\mathsf{RFN}}

\newcommand{\Con}{\mathsf{Con}}

\newcommand{\B}{\mathsf{B}}

\newcommand{\hyp}{\mathchar`-}

\newcommand{\Th}{\mathrm{Th}}
\newcommand{\Rfn}{\mathsf{Rfn}}

\begin{document}

\title{On some subtheories of Strong Dependent Choice}
\author{JUAN P. AGUILERA}
\address{Institute~of~discrete~mathematics~and~geometry, TU~Wien. Wiedner~Hauptstrasse~8-10, 1040~Vienna, Austria.}
\email{aguilera@logic.at}
\author{YUDAI SUZUKI}
\address{National~Institute~of~Technology, Oyama~College, Nakakuki~771, Oyama, Tochigi, Japan}
\email{yudai.suzuki.research@gmail.com}
\author{KEITA YOKOYAMA}
\address{Mathematical~Institute, Tohoku~University, Aramaki~Aza-Aoba~6-3, Aoba-ku, Sendai, Miyagi, Japan}
\email{keita.yokoyama.c2@tohoku.ac.jp}

\date{\today}

\begin{abstract}
  In this paper, we give characterizations of
  the set of $\Pi^1_{e}$-consequences, $\Sigma^1_{e}$-consequences and $\B(\Pi^1_{e})$-consequences of
  the axiomatic system of the strong dependent choice for $\Sigma^1_i$ formulas $\Sigma^1_i\hyp\SDC_0$ for $i > 0$ and $e < i+2$. 
  Here, $\B(\Gamma)$ denotes the set generated by $\land,\lor,\lnot$ starting from $\Gamma$.
\end{abstract}

\maketitle

\section{Introdcution}
In the proof theoretic study of formal systems of arithmetic,
we are interested  not only in  the whole set of the sentences provable from a theory, but also some certain subsets of it.
For example, the consistency strength of a theory is characterized by its $\Pi^0_1$-consequences,
the proof theoretic ordinal is naturally defined by its $\Pi^1_1$-consequences, and
its $\Pi^1_2$-consequences corresponds to its proof theoretic dilator.

It is also important in reverse mathematics to study the set of consequences of a theory with restricted complexity.
Recently, the $\Pi^1_2$-consequences of $\Pi^1_1\hyp\CAo$ have been studied thoroughly
since $\Pi^1_1\hyp\CAo$ is not axiomatizable by any $\Pi^1_2$ sentence.
The first giant leap for the study of $\Pi^1_2$-consequences of $\Pi^1_1\hyp\CAo$ is achieved by Towsner \cite{Townser_TLPP}.
He introduced new $\Pi^1_2$ principles whose logical strength are between $\Pi^1_1\hyp\CAo$ and $\ATRo$.
The second and third author extended Towsner's work and gave a hierarchy which covers the set of $\Pi^1_2$ consequences of $\Pi^1_1\hyp\CAo$ \cite{suzuki_yokoyama_pi12}.

In this paper, we generalize their work in \cite{suzuki_yokoyama_pi12}.
For more detail, we will give a characterization of the $\Pi^1_e$-, $\Sigma^1_e$- and $\B(\Pi^1_e)$-consequences of the system
strong $\Sigma^1_i$ dependent choice ($\Sigma^1_i\hyp\SDC_0$) where
$\B(\Pi^1_e)$ is the set of boolean combinations of $\Pi^1_e$ formulas.
Since $\Sigma^1_i\hyp\SDC_0$ is $\Pi^1_4$-conservative over $\Pi^1_i\hyp\CAo$, we also have a characterization of
$\Pi^1_e$-, $\Sigma^1_{e'}$- and $\B(\Pi^1_{e'})$-consequences of $\Pi^1_i\hyp\CAo$ for $e \leq 4$ and $e' \leq 3$.
We note that the characterization of $\B(\Pi^1_e)$-consequences can be seen as an analogue of the study of the syntactic local reflection principles in first-order arithmetic by Beklemishev \cite{Beklemishev_local_reflection}.

We give a general idea used throughout this paper. 
It is known that $\Sigma^1_i\hyp\SDC_0$ is characterized by the $\beta_i$-model reflection principle (see Theorem \ref{thm equiv of SDC and beta}).
Therefore, by extracting a suitable property of $\beta_i$-models, we will get a corresponding subtheory of $\Sigma^1_i\hyp\SDC_0$.
For proving a correspondence, we use the Barwise-Schlipf argument.
That is, we make a model $\M$ from a given $T \subseteq \Sigma^1_i\hyp\SDC_0$ and a $T$ non-provable sentence $\sigma$ such that
 $\M$ has a $\beta_i$-submodel $\M'$ satisfying $\Sigma^1_i\hyp\SDC_0$
by the compactness theorem.
Then, $\M'$ works as a witness for the non-provability of $\sigma$ from $\Sigma^1_i\hyp\SDC_0$.

In Section 2, we give a characterization of the $\Pi^1_e$-consequences of $\Sigma^1_{i}\hyp\SDC_0$ for $i > 0$ and $e < i +2$.
We note that the main result of this section (Theorem \ref{thm main result of sec 2}) is essentially mentioned in
\cite{pacheco2022determinacy}, but our proof is more informative for computability-theoretic aspects (Lemma \ref{lemma, number of beta}).
We also give a characterization of the $\B(\Pi^1_{i+1})$- and $\Sigma^1_{i+1}$-consequences of $\Sigma^1_i\hyp\SDC_0$ in Section 3.
We then combine this result and the results in Section 2, and give a finer analysis for the $\Pi^1_{e}$-consequences of $\Sigma^1_i\hyp\SDC_0$ for
$e \leq i$.
By modifying the proofs in Section 3, we give a characterization of the $\B(\Pi^1_{e+1})$- and $\Sigma^1_{e+1}$-consequences of $\Sigma^1_i\hyp\SDC_0$ for $e < i$ in Section 4.

\subsection*{Acknowledgements}
Aguilera would like to acknowledge support from the Austrian Science Foundation through grant STA-139.
Yokoyama's work is partially supported by JSPS KAKENHI grant numbers JP19K03601, JP21KK0045 and JP23K03193.

\section{$\Pi^1_e$ formulas provable from $\Sigma^1_i\hyp\SDC_0$}
In this section, we give a characterization of the set $\{\sigma \in \Pi^1_e :\Sigma^1_i\hyp\SDC_0 \vdash \sigma\}$
for $i > 0, e < i +2$.
As noted in the introduction, we use a characterization of $  \Sigma^1_i\hyp\SDC$ based on coded $\beta_i$-models for this characterization.
We begin with introducing $\Sigma^1_i\hyp\SDC_0$.
We refer the reader to Simpson's textbook \cite{Simpson} for the basic definitions.
We note that in \cite{Simpson}, the system $\Sigma^1_i\hyp\SDC_0$ is denoted
as $\mathsf{Strong}\, \Sigma^1_i\hyp\mathsf{DC}_0$.

\begin{definition}[sequence of sets, $\RCAo$]\label{def seq of sets}
  Let $(\bullet,\bullet) :\N^2 \to \N$ be a fixed primitive recursive pairing function.
  For $X \subseteq \N$ and $x \in \N$, we define the $x$-th segment $X_x = \{y : (x,y) \in \N\}$.
  Then, $X$ is regarded as the sequence $\langle X_x \rangle_{x \in \N}$.
  In this sense, we write $A \in X$ to mean $\exists x(A = X_x)$.
  For a sequence $\langle X_n \rangle_{n \in \N}$ of sets, we define
  $X_{<k}$ as the sequence $\langle X_0,\ldots,X_{k-1} \rangle$.
\end{definition}

\begin{definition}
  Let $\varphi$ be a formula. We define strong dependent choice for $\varphi$ as follows.
  \begin{align*}
    \varphi\hyp\SDC \equiv \forall X \exists \langle Y_n \rangle_n
    (Y_0 = X \land \forall n >0 (\exists Z \varphi(Y_{<n},Z) \to \varphi(Y_{<n},Y_{n}))).
  \end{align*}
  We define the axiomatic system $\Sigma^1_i\hyp\SDC_0$ as $\ACA_0 + \{\varphi\hyp\SDC : \varphi \in \Sigma^1_i\}$.
\end{definition}

We next see the definition of $\beta_i$-models, $\beta_i$-submodels and coded $\beta_i$-models.
Intuitively, a $\beta_i$-model is a $\Sigma^1_i$-elementary submodel of $(\omega,\mathcal{P}(\omega))$.

\begin{definition}
  An $\L_2$ structure $\M$ of the form $\M = (\omega,S)$ is called an $\omega$-model.
  Let $\M=(\omega,S)$ be an $\omega$-model and $i \in \omega$. We say $\M$ is a $\beta_i$-model if
    $\M \models \sigma \Leftrightarrow (\omega,\mathcal{P}(\omega)) \models \sigma$
     for any $\Sigma^1_i$ formula $\sigma$ with parameters from $\M$.
\end{definition}

The notions of $\omega$-models and $\beta_i$-models can be extended to a relation of two $\L_2$ structures
who share the first-order part.

\begin{definition}
  Let $\M = (\N^{\M},S^{\M})$ and $\M' = (\N^{\M'},S^{\M'})$ be $\L_2$ structures.
  We say $\M'$ is a $\omega$-submodel of $\M$ if $\N^{\M'} = \N^{\M}$ and $S^{\M'} \subseteq S^{\M}$.

  Let $i \in \omega$.
  We say $\M'$ is a $\beta_i$-submodel of $\M$ if $\M'$ is an $\omega$-submodel of $\M$ and
  $\M' \models \sigma \Leftrightarrow \M \models \sigma$
  for any $\Sigma^1_i$ formula $\sigma$ with parameters from $\M'$.
\end{definition}

If $S = \{S_n : n \in \omega\} \subseteq \mathcal{P}(\omega)$ is countable,
then $S$ can be coded by a single set $S_C \subseteq \omega$ by putting $S_C = \{(n,s) : s \in S_n\}$ as in Definition \ref{def seq of sets}.
This idea can be generalized to the nonstandard models.
We call a model $\M=(\N^{M},S_C)$ is a coded $\omega$-submodel in $\M' = (\N^{\M'},S')$ 
when $\N^{\M} = \N^{\M'}$ and $S_C \in S'$.
If $\M$ and $\M'$ satisfy the same $\Sigma^1_i$ sentences, then $\M$ is called a coded $\beta_i$-submodel in $\M'$.
Throughout this paper, we omit the word \textit{coded} if it is clear from the context.

\begin{definition}[$\ACAo$]
  Let $\M = \langle M_x \rangle_x$ be a coded $\omega$-model.
  We say $\M$ is a $\beta_i$-model if
  \begin{align*}
    \sigma \leftrightarrow \M \models \sigma
  \end{align*}
  for any $\Pi^1_i$ formula $\sigma$ having parameters from $\M$.
\end{definition}
\begin{remark}
  We give a few comments on the above definition.
  \begin{enumerate}
    \item In this paper, we may think that $\M \models \sigma$ is the abbreviation of `the relativization $\sigma^{\M}$ is true' because 
    our base theory is sufficiently strong. For the difference of $\M \models \sigma$ in the sense that an evaluation for $\sigma$ exists and $\sigma^{\M}$, see Simpson's textbook and Remark 2.15 in \cite{suzuki_yokoyama_pi12}.
    \item A $\beta_0$-model is just the same as an $\omega$-model.
    \item Let $i > 0$. Then, the condition $\M$ is a $\beta_i$-model is equivalent to
    \begin{align*}
      (\M \models \sigma) \to \sigma  \text{ for any } \Pi^1_{i} \text{ formula.}
    \end{align*}
    Therefore, using a universal $\Pi^1_{i}$ formula, we can find a $\Pi^1_i$ formula stating that `$\M$ is a $\beta_i$-model'.
    Throughout this paper, we fix a $\Pi^1_i$ formula $\beta_i(\M)$ stating `$\M$ is a $\beta_i$-model'.
  \end{enumerate}
\end{remark}

We recall that $\Sigma^1_i\hyp\SDC_0$ is characterized by $\beta_i$-models.
\begin{theorem}\label{thm equiv of SDC and beta}
  \cite[VII.7.4. and VII.7.6.]{Simpson}
  For $i> 0$, the following assertions are equivalent over $\ACAo$.
  \begin{enumerate}
    \item $\Sigma^1_i\hyp\SDC_0$,
    \item the $\beta_i$-model reflection: $\forall X \exists \M (X \in \M \land \beta_i(\M))$,
    \item $\forall X(\sigma(X) \to \exists \M(X \in \M \land \beta_i(\M) \land \M \models \sigma(X)))$ for $\sigma \in \Sigma^1_{i+2}$.
  \end{enumerate}
\end{theorem}
Following these equivalences, it is shown that
$\Sigma^1_i\hyp\SDC_0$ is axiomatizable by a $\Pi^1_{i+2}$ sentence but
is not axiomatizable by any $\Sigma^1_{i+2}$ sentence.
In fact, it follows from $(3)$ that $\Sigma^1_i\hyp\SDC_0$ proves
$\sigma \to \Con(\ACAo + \sigma)$ for any $\Sigma^1_{i+2}$ sentence $\sigma$ where $\Con(\ACAo + \sigma)$ is the $\Pi^0_1$ formula asserting the consistency of $\ACAo + \sigma$.
Hence, if $\Sigma^1_i\hyp\SDC_0$ were axiomatized by a $\Sigma^1_{i+2}$ sentence $\sigma$,
then $\sigma$ would prove $\Con(\ACAo + \sigma)$. However, this contradicts G\"{o}del's second incompleteness theorem.

Since $\Sigma^1_i\hyp\SDC_0$ is finitely axiomatizable and not axiomatizable by any $\Sigma^1_{i+2}$ sentence,
the set $\Th_{\Pi^1_e}( \Sigma^1_i\hyp\SDC_0)$
 of $\Pi^1_{e}$-consequences of $ \Sigma^1_i\hyp\SDC_0$ is strictly weaker than $\Sigma^1_i\hyp\SDC_0$ for $e \leq i+1$. 

In the remaining of this section, we give a characterization of the set
$\Th_{\Pi^1_{e+2}}( \Sigma^1_i\hyp\SDC_0)$ based on a weaker variant of (2) in Theorem \ref{thm equiv of SDC and beta}
for $i > 1$ and $e \leq i+1$.
We note that $\Th_{\Pi^1_2}(\Sigma^1_1\hyp\SDC_0)$, which is equal to $\Th_{\Pi^1_2}(\Pi^1_1\hyp\CAo)$, has already been studied in \cite{suzuki_yokoyama_pi12}.

\begin{definition}
  Let $n > 0$.
  We define a formula $\beta^i(\M_0,\ldots,\M_n)$ by
  \begin{align*}
    \bigwedge_{m < n} (\M_m \in \M_{m+1}) \land \bigwedge_{m < n} \M_{n} \models \beta_i(\M_m).
  \end{align*}
  We also define a formula $\beta^i_e(\M_0,\ldots,\M_n)$ by $\beta^i(\M_0,\ldots,\M_n) \land \beta_e(\M_n)$.
\end{definition}
We remark that if $\beta^i(\M_0,\ldots,\M_n)$, then $\M_{m'} \in \M_m$ and
$\M_m \models \beta_i(\M_{m'})$ for any $m' < m < n+1$.
If $\beta_e(\M_n)$ in addition, then $\beta_e(\M_m)$ for any $e \leq i$ and $m < n$.
In particular, $\beta^i_i(\M_0,\ldots,\M_n)$ means that
$\M_0 \in \cdots \in \M_n$ and $\beta_i(\M_m)$ for all $m$.

\begin{definition}
  Let $\sigma,\tau$ be sentences.
  For $i,e,n \in \omega$, we define $\beta^i_e\RFN(\sigma;n;\tau)$ as follows.
  \begin{align*}
    \beta^i_e\RFN(\sigma; n; \tau) \equiv &\forall X \exists \M_0,\ldots,\M_n \text{ with the following properties. } 
  \end{align*}
  \begin{itemize}
    \item $X \in \M_0$,
    \item $\beta^i_e(\M_0,\ldots,\M_n)$,
    \item $\M_m \models \ACAo$ for all $m < n+1$,
    \item $\M_0 \models \sigma$ and $\M_n \models \tau$.
  \end{itemize}
  We sometimes omit $\sigma$ or $\tau$ when they are trivial. For example,
  we write $\beta^i_e\RFN(n;\tau)$
  instead of $\beta^i_e\RFN(0=0;n; \tau)$.
\end{definition}
\begin{remark}
  When we work in a theory including $\ACAo$ and $e > 0$, then
  we may omit the condition $\M_m \models \ACAo$ from the above definition because
  $\ACAo$ proves that any $\beta_e$-model is a model of $\ACAo$.
\end{remark}

\begin{remark}
We note that $\beta^i_e\RFN(n)$ is first introduced in \cite[Section 2]{pacheco2022determinacy} as $\psi_e(i,n)$.
In that paper, it is shown that the syntactic reflection of $ \Sigma^1_i\hyp\SDC_0$ for $\Pi^1_{e+2}$ formulas is equivalent to $\forall n. \psi_e(i,n)$ over $\ACAo$.
In this section, we consider a standard analogue of this result.
We also note that the concept of $\beta^i_e\RFN(n;\tau)$ is introduced in \cite{suzuki_yokoyama_pi12}.
\end{remark}

By iterating (2) in Theorem \ref{thm equiv of SDC and beta} finitely many times,
$\Sigma^1_i\hyp\SDC_0$ proves
\begin{align*}
  \forall X \exists \M_0,\ldots,\M_n
  (X \in \M_0 &\land \beta^i_i(\M_0,\ldots,\M_n))
\end{align*}
for any $n \in \omega$.
Notice that $\beta^i_e\RFN(n)$ is a $\Pi^1_{e+2}$ approximation of the above sentence.

In \cite{suzuki_yokoyama_pi12}, it is proved that
the set $\{\beta^1_0\RFN(n) : n \in \omega\}$ characterizes the $\Pi^1_2$-consequences of $\Pi^1_1\hyp\CAo$.
The key point of this proof is that when a $\Pi^1_3$ sentence of the form $\forall X \exists Y \theta(X,Y)$ is provable from
$\Pi^1_1\hyp\CAo$, then there exists $n \in \omega$ such that
a solution $Y$ is computable from any $\M_0,\ldots,\M_n$ with $\beta^1_1(\M_0,\ldots,\M_n)$ and $X \in \M_0$.
We extend this result to more general cases.

\begin{lemma}\label{lem : closed under beta_i}
  Let $i \in \omega$ such that $i \geq 1$ and $\M = (\N^{\M},S^{\M})$ be a model of $\ACAo$.
  Let $S' \subseteq S^{\M}$ be such that
  $\M' = (\N^{\M},S')$ is a model of $\RCAo$ with the following property:
  for any $X \in S'$, there exists $Y \in S'$ such that
  $\M \models \beta_i(Y) \land (X \in Y)$.
  Then,
  \begin{itemize}
    \item $\M \models \beta_i(Y)$ implies $\M' \models \beta_i(Y)$ for any $Y \in S'$,
    \item $\M'$ is a model of $  \Sigma^1_i\hyp\SDC_0$, and
    \item $\M'$ is a $\beta_i$-submodel of $\M$.
  \end{itemize}
\end{lemma}
\begin{proof}
  We prove the lemma by induction. The case $i = 1$ is proved in \cite{suzuki_yokoyama_pi12}.

  Let $i \in \omega$. Assume that for any $X \in S'$, there exists $Y \in S'$ such that
  $\M \models \beta_{i+1}(Y) \land (X \in Y)$.
  By the induction hypothesis, the following properties hold.
  \begin{enumerate}
    \item $\M \models \beta_i(Y)$ implies $\M' \models \beta_i(Y)$ for any $Y \in S'$, and
    \item $\M'$ is a $\beta_i$ submodel of $\M$.
  \end{enumerate}
  Let $Y \in S'$ such that $\M \models \beta_{i+1}(Y)$. We show that
  $\M' \models \beta_{i+1}(Y)$.
  Let $\varphi(X,Z)$ be a $\Pi^1_i$ formula and $X \in Y$ such that
  $\M' \models \exists Z\varphi(X,Z)$.
  Since $\M'$ is a $\beta_i$ submodel of $\M$, $\M \models \exists Z\varphi(X,Z)$.
  Then, $\M \models [Y \models \exists Z\varphi(X,Z)]$ because $\M \models \beta_{i+1}(Y)$.
  That is, there exists $Z \in Y$ such that $\M \models [Y \models \varphi(X,Z)]$.
  Then, $\M' \models [Y \models \varphi(X,Z)]$ because $Y \models \varphi(X,Z)$ is an arithmetical condition for $X,Y$ and $Z$.
  Thus, $\M' \models [Y \models \exists Z \varphi(X,Z)]$.

  We next show that $\M'$ is a $\beta_{i+1}$ submodel of $\M$.
  Let $\varphi(X,Z)$ be a $\Pi^1_i$ formula and $X \in \M'$ such that
  $\M \models \exists Z\varphi(X,Z)$.
  Let $Y \in \M'$ such that
  $\M \models \beta_{i+1}(Y) \land X \in Y$.
  Then, $\M \models [Y \models \exists Z\varphi(X,Z)]$ and hence
  $\M' \models [Y \models \exists Z\varphi(X,Z)]$.
  Therefore $\M' \models \exists Z \varphi(X,Z)$ because $\M' \models \beta_{i}(Y)$ by the induction hypothesis.

  Since $\M' \models \forall X \exists Y (X \in Y \land \beta_{i+1}(Y))$, $\M'$ is a model of
  $ \Sigma^1_{i+1}\hyp\SDC_0$.
\end{proof}

\begin{lemma}\label{lemma, number of beta}
  Let $\varphi$ be a $\Pi^1_i$ formula such that
  $\forall X \exists Y \varphi$ is provable from $ \Sigma^1_i\hyp\SDC_0$.
  Then, there exists an $n \in \omega$ such that $\ACAo$ proves that
  \begin{align*}
    \forall X,\M_0,\ldots,\M_n((X \in \M_0 \land \beta^i_i(\M_0,\ldots,\M_n)) \to \exists Y \in \M_n \varphi(X,Y)).
  \end{align*}
\end{lemma}
\begin{proof}
  Let $\varphi$ be a $\Pi^1_i$ formula such that
  $  \Sigma^1_i\hyp\SDC_0 \vdash \forall X \exists Y \varphi$.
  Assume that $\ACAo$ does not prove
  \begin{align*}
    \forall X,\M_0,\ldots,\M_n((X \in \M_0 \land \beta^i_i(\M_0,\ldots,\M_n)) \to \exists Y \in \M_n \varphi(X,Y))
  \end{align*}
  for any $n \in \omega$. 
  That is, for any $n \in \omega$, $\ACAo$ is consistent with
  \begin{align*}
    \exists X,\M_0,\ldots,\M_n
    (&X \in \M_0 \land \beta^i_i(\M_0,\ldots,\M_n) \land \forall Y \in \M_n \lnot \varphi(X,Y)).
  \end{align*}
  Let $\L'$ be the language given by adding new set constant symbols $C$ and $D_0,D_1,\ldots$ into $\L_2$.
  Define an $\L'$-theory $T$ consisting of the following sentences:
  \begin{align*}
    &\ACA_0, \\
    &C \in D_0, \\
    &D_n \in D_{n+1} \text{ for } n \in \omega, \\
    &\beta_i(D_n) \text{ for } n \in \omega,  \\
    &\forall Y \in D_n \lnot \varphi(C,Y)  \text{ for } n \in \omega.
  \end{align*}
  Then, each finite subset of $T$ is consistent.
  Thus, from the compactness theorem, there exists a model $\M = (\N^{\M},S^{\M},C,\{D_n\}_{n \in \omega})$ of $T$. Define $S' \subseteq S^{\M}$ by
  $S' = \bigcup_{n \in \omega }\{X \in S^{\M} : \M \models X \in D_n \}$.

  It is easy to see that $\M' = (\N^{\M},S')$ is a model of $\RCAo$.
  Moreover, for any $X \in S'$, there exists $Y \in S'$ such that
  $\M \models X \in Y \land \beta_i(Y)$.
  Thus, by Lemma \ref{lem : closed under beta_i},
  $\M'$ is a model of $  \Sigma^1_i\hyp\SDC_0$ and a $\beta_i$ submodel of $\M$.
  Thus $\M' \models \forall Y \lnot \varphi(C,Y)$ because $\M \models \lnot \varphi(C,Y)$ for any $Y \in \M'$.
  However, this contradicts the assumption $  \Sigma^1_i\hyp\SDC_0 \vdash \forall X \exists Y \varphi(X,Y)$.
\end{proof}

\begin{theorem}\label{thm main result of sec 2}
  Let $e,i \in \omega$ such that $i > e$.
  Then, the theories $\ACA_0 + \{\beta^i_e\RFN(n) : n \in \omega\}$ and
  $\Th_{\Pi^1_{e+2}}( \Sigma^1_{i}\hyp\SDC_0)$ prove
  the same sentences.
\end{theorem}
\begin{proof}
  Since $\ACA_0$ is $\Pi^1_2$ axiomatizable, $\ACAo$ is in $\Th_{\Pi^1_{i+2}}( \Sigma^1_{i}\hyp\SDC_0)$.
  It is easy to see that $\beta^i_e\RFN(n)$ is in $\Th_{\Pi^1_{i+2}}( \Sigma^1_{i}\hyp\SDC_0)$ for each $n \in \omega$.
  Therefore, any sentence provable from $\ACA_0 + \{\beta^i_e\RFN(n) : n \in \omega\}$ is also provable from
   $\Th_{\Pi^1_{i+2}}( \Sigma^1_{i}\hyp\SDC_0)$.

   For the converse, let $\sigma$ be a $\Pi^1_{e+2}$ sentence provable from $ \Sigma^1_{i}\hyp\SDC_0$.
   Write $\sigma \equiv \forall X \exists Y \varphi(X,Y)$ by a $\Pi^1_e$ formula $\varphi$.
   Then, there exists $k \in \omega$ such that $\ACAo$ proves
   \begin{align*}
    \forall X,\M_0,\ldots,\M_n((X \in \M_0 \land \beta^i_i(\M_0,\ldots,\M_n)) \to \exists Y \in \M_n \varphi(X,Y)).
   \end{align*}
   We show that $\forall X \exists Y \varphi(X,Y)$ is provable from $\ACA_0 + \{\beta^i_e\RFN(n) : n \in \omega\}$.
   We reason in $\ACA_0 + \{\beta^i_e\RFN(n) : n \in \omega\}$.
   Take an arbitrary $X$. By $\beta^i_e\RFN(m+1)$, take $\M_0,\ldots,\M_{k+1}$ such that
   \begin{align*}
     X \in \M_0 \land \beta^i_e(\M_0,\ldots,\M_{k+1}) \land \bigwedge_{m < k+2} \M_m \models \ACAo.
   \end{align*}
   Then, $\M_{k+1}$ satisfies $\ACAo \land X \in \M_0 \land \beta^i_i(\M_0,\ldots,\M_k)$.
   Thus, $\M_{k+1} \models \exists Y \in \M_k \varphi(X,Y)$ and hence
   $\M_{k+1} \models \exists Y \varphi(X,Y)$.
   Since $\M_{k+1}$ is a $\beta_e$ model and $\varphi$ is a $\Pi^1_e$ formula,
   $\exists Y \varphi(X,Y)$ holds.
\end{proof}

\begin{corollary}
  $\ACA_0 + \{\beta^i_e\RFN(n) : n \in \omega\}$ is  the same as
  $\Th_{\Pi^1_{e+2}}(\Pi^1_i\hyp\CAo)$ for $e=0,1,2$ and $i > e$.
\end{corollary}
\begin{proof}
  This is immediate from the fact that $ \Sigma^1_i\hyp\SDC_0$ is $\Pi^1_4$-conservative over $\Pi^1_i\hyp\CAo$ \cite[VII.6.20]{Simpson}.
\end{proof}

\section{Boolean combinations of $\Pi^1_{i+1}$ formulas provable from $  \Sigma^1_i\hyp\SDC$}
In this section, we give a characterization of the set $\{\sigma \in \B(\Pi^1_{i+1}) :   \Sigma^1_i\hyp\SDC_0 \vdash \sigma\}$
for $i > 0$.
Here, $\B(\Gamma)$ is the set of boolean combinations of formulas in $\Gamma$.
In other words,
$\B(\Gamma)$ denotes the set generated by $\land, \lor, \lnot$ starting from $\Gamma$.

\begin{definition}
  For a sentence $\sigma$, we define $\beta_i\Rfn(\sigma)$ by
  \begin{align*}
    \sigma \to \exists \M (\beta_i(\M) \land \M \models \sigma).
  \end{align*}
  For a class $\Gamma$ of formulas, we define an axiomatic system
  $\beta_i\Rfn(\Gamma)_0$ by $\ACAo + \{\beta_i\Rfn(\sigma) : \sigma \in \Gamma, \sigma \text{ is a sentence}\}$.
\end{definition}
\begin{remark}
  Notice that $\beta_i\Rfn(\Sigma^1_{i+1})_0$
  is a weaker variant of (3) in Theorem \ref{thm equiv of SDC and beta}.
  The difference between those is whether $\sigma$ allows (set) parameters.
  Since each instance of $\beta_i\Rfn(\Sigma^1_{i+1})_0$ is $\Sigma^1_{i+2}$, $\beta_i\Rfn(\Sigma^1_{i+1})_0$ is strictly weaker than $\Sigma^1_{i}\hyp\SDC_0$.
\end{remark}

We show that the theories 
$\beta_i\Rfn(\Sigma^1_{i+1})_0$ and 
$\Th_{\B(\Pi^1_{i+1})}(\Sigma^1_{i}\hyp\SDC_0)$ are equivalent.
We start with showing a basic property of $\B(\Pi^1_{i})$.

\begin{lemma}\label{CNF}
  Let $\rho \in \B(\Pi^1_i)$. Then, there exist $\pi_0,\ldots,\pi_{n-1} \in \Pi^1_i$ and $\sigma_0,\ldots,\sigma_{n-1} \in \Sigma^1_i$ such that
  $\rho$ is equivalent to $\bigwedge_{n' < n}(\pi_{n'} \lor \sigma_{n'})$ over $\RCAo$.
\end{lemma}
\begin{proof}
  We reason in $\RCAo$. Then, $\Pi^1_i$ and $\Sigma^1_i$ are closed under $\lor$.

  Let $\rho$ be a $\B(\Pi^1_i)$ sentence.
  Taking a conjunction normal form,
  \begin{align*}
    \rho \leftrightarrow \bigwedge_{n' < n} \bigvee_{m' < m} (\pi_{n',m'} \lor \sigma_{n',m'})
  \end{align*}
  where each $\pi_{n',m'}$ is $\Pi^1_i$ and each $\sigma_{n',m'}$ is $\Sigma^1_i$.
  Put $\pi_{n'} \equiv \displaystyle{\bigvee_{m'}} \pi_{n',m'}$ and $\sigma_{n'} \equiv \displaystyle{\bigvee_{m'}} \sigma_{n',m'}$.
  Then, $\pi_{n'} \in \Pi^1_i$, $\sigma_{n'} \in \Sigma^1_i$ and
  \begin{align*}
    \rho \leftrightarrow \bigwedge_{n' < n} (\pi_{n'} \lor \sigma_{n'}).
  \end{align*}
    This completes the proof.
\end{proof}

\begin{definition}
  We define $\beta^i_e\RFN^-(\varphi;n)$ as 
  the $\Sigma^1_{i+1}$ sentence stating that there exist $Y_0,\ldots Y_n$ such that
$Y_0 \models \varphi$, and $\beta^i_e(Y_0,\ldots,Y_n)$. We omit $\varphi$ if it is trivial.
\end{definition}

\begin{lemma}\label{lem iteration of beta-model}
  Let $i \in \omega$ such that $i > 0$.
  Then, for any $n \in \omega$,
  $\beta_i\Rfn(\Sigma^1_{i+1})_0$ proves $\varphi \to \beta^i_i\RFN^-(\varphi;n)$ for
  any sentence $\varphi \in \Sigma^1_{i+1}$.
\end{lemma}
\begin{proof}
  We prove the lemma by induction.
  In the case $n = 0$, $\varphi \to \beta^i_i\RFN^-(\varphi;n)$ is the same as $\varphi \to \exists Y(\beta_i(Y) \land Y \models \varphi)$.
  Therefore, it is an instance of $\beta_i\Rfn(\Sigma^1_{i+1})_0$ and hence provable from $\beta_i\Rfn(\Sigma^1_{i+1})_0$.

  For the induction step, take an arbitrary $n$ and assume that
  $\beta_i\Rfn(\Sigma^1_{i+1})_0$ proves $\varphi \to \beta^i_i\RFN^-(\varphi;n)$ for any sentence $\varphi \in \Sigma^1_{i+1}$.
  Take an arbitrary sentence $\varphi \in \Sigma^1_{i+1}$.
  We reason in $\beta_i\Rfn(\Sigma^1_{i+1})_0$ and show that $\varphi \to \beta^i_i\RFN^-(\varphi;n+1)$ holds.

  Assume  that $\varphi$ holds.
  Then, we have $\beta^i_i\RFN^-(\varphi;n)$ because of the induction hypothesis.
  Since $\beta^i_i\RFN^-(\varphi;n)$ is a $\Sigma^1_{i+1}$ sentence,
  \begin{align*}
    \beta^i_i\RFN^-(\varphi;n) \to \exists Y(\beta_i(Y) \land Y \models \beta^i_i\RFN^-(\varphi;n))
  \end{align*}
  also holds by the induction hypothesis. Thus we have $\exists Y(\beta_i(Y) \land Y \models \beta^i_i\RFN^-(\varphi;n))$.
  Take a $\beta_i$-model $Y$ such that $Y \models \beta^i_i\RFN^-(\varphi;n)$.
  Then, there exist $Y_0,\ldots,Y_n$ such that
  \begin{align*}
    Y \models \left[(Y_0 \models \varphi) \land \bigwedge_{m < n+1} \beta_i(Y_m) \land \bigwedge_{m < n} Y_m \in Y_{m+1} \right].
  \end{align*}
  Since the condition in the square bracket is $\Pi^1_i$,
  \begin{align*}
    (Y_0 \models \varphi) \land \bigwedge_{m < n+1} \beta_i(Y_m) \land \bigwedge_{m < n} Y_m \in Y_{m+1}
  \end{align*}
  actually holds. Take $Y_{n+1} = Y$, then $Y_0,\ldots,Y_n,Y_{n+1}$ witness
  $\beta^i_i\RFN^-(\varphi;n+1)$.
\end{proof}

\begin{lemma}\label{Lem_cons_or}
  Let $i\in \omega, i > 0$. Let
 $\pi$ be a $\Pi^1_{i+1}$ sentence and $\sigma$ a $\Sigma^1_{i+1}$ sentence such that
  $\Sigma^1_{i}\hyp\SDC_0 \vdash \pi \lor \sigma$.
  Then, $\beta_i\Rfn(\Sigma^1_{i+1})_0 \vdash
  \pi \lor \sigma$.
\end{lemma}
\begin{proof}
  For the sake of contradiction, assume that $\Sigma^1_i\hyp\SDC_0$ proves $\pi \lor \sigma$ but $\beta_i\Rfn(\Sigma^1_{i+1})_0$ does not prove it.
  Then, the theory $T = \beta_i\Rfn(\Sigma^1_{i+1})_0 \cup \{\lnot \pi,  \lnot \sigma\}$
  is consistent.

  Let $\M = (\N^{\M},S^{\M})$ be a model of $T$.
  We take new set constant symbols $C_n$ for each $n \in \omega$.
  Define a theory $T'$ by
  \begin{align*}
    T' = \beta_i\Rfn(\Sigma^1_{i+1})_0 &+ \lnot \sigma + \{\beta^i_i(C_0,\ldots,C_n) : n \in \omega\} +  C_0 \models \lnot \pi .
  \end{align*}
  Then, $\M$ satisfies every finite subtheory of $T'$ because $\M$ is a model of $\beta_i\Rfn(\Sigma^1_{i+1})_0$, $\lnot \sigma$ and $\beta^i_i\RFN^-(\lnot \pi;n)$ for all $n$.
  Therefore, $T'$ is a consistent theory.

  Let $\M' = (\N^{\M'},S,\{Y_n : n \in \omega\})$ be a model of $T'$.
  Define $S' = \bigcup_{n \in \omega}\{X \in S : \M' \models X \in Y_n\}$.
  Put $\M'' = (\N^{\M'},S')$. Then, by Lemma \ref{lem : closed under beta_i},
  \begin{itemize}
    \item $\M''$ is a model of $\Sigma^1_i\hyp\SDC_0$,
    \item $\M''$ is a $\beta_i$-submodel of $\M'$ and
    \item for any $n \in \omega $, $\M'' \models \beta_i(Y_n)$.
  \end{itemize}
  We note that $\M''$ satisfies $\lnot \sigma$ because $\M'$ satisfies $\lnot \sigma$ and $\M''$ is a $\beta_i$-submodel of $\M'$.
  We claim that $\M''$ also satisfies $\lnot \pi$.
  Indeed, $\M'' \models [Y_0 \models \lnot \pi]$ because $\M' \models [Y_0 \models \lnot \pi]$ and the condition $Y_0 \models \lnot \pi$ is arithmetical. Since $\M'' \models \beta_i(Y_0)$, $\M'' \models \lnot \pi$.

  Consequently, $\M''$ is a model of $\Sigma^1_i\hyp\SDC_0 \cup \{ \lnot \pi, \lnot \sigma\}$. However, this contradicts the assumption $\Sigma^1_i\hyp\SDC_0$ proves $\pi \lor \sigma$.
\end{proof}

\begin{theorem}
  The theories $\beta_i\Rfn(\Sigma^1_{i+1})_0$ and $\Th_{\B(\Pi^1_{i+1})}(\Sigma^1_i\hyp\SDC_0)$ prove the same sentences.
\end{theorem}
\begin{proof}
  It is enough to show that each axiom in $\beta_i\Rfn(\Sigma^1_{i+1})_0$ is provable from $\Th_{\B(\Pi^1_{i+1})}(\Sigma^1_i\hyp\SDC_0)$ and vice verse.
  It is immediate that each axiom in $\beta_i\Rfn(\Sigma^1_{i+1})_0$ is provable from $\Th_{\B(\Pi^1_{i+1})}(\Sigma^1_i\hyp\SDC_0)$. We show that
  each $\B(\Pi^1_{i+1})$ formula provable from $\Sigma^1_i\hyp\SDC_0$ is already provable from
  $\beta_i\Rfn(\Sigma^1_{i+1})_0$.

  Let $\rho \in \Th_{\B(\Pi^1_{i+1})}(\Sigma^1_i\hyp\SDC_0)$.
  We  show that $\rho$ is provable from $\beta_i\Rfn(\Sigma^1_{i+1})_0$.
  There exist $\pi_0,\ldots,\pi_{n-1} \in \Pi^1_{i+1}, \sigma_0,\ldots,\sigma_{n-1} \in \Sigma^1_{i+1}$ such that
  $\rho$ is equivalent to  $\bigwedge_{m < n} \pi_m \lor \sigma_m$ by Lemma \ref{CNF}.
  Since $\rho$ is provable from $\Sigma^1_i\hyp\SDC_0$, $\pi_m \lor \sigma_m$ is provable from $\Sigma^1_i\hyp\SDC_0$ for any $m < n$.
  Therefore, $\beta_i\Rfn(\Sigma^1_{i+1})_0$ proves $\pi_m \lor \sigma_m$ for any $m < n$ by Lemma
  \ref{Lem_cons_or}. Hence $\beta_i\Rfn(\Sigma^1_{i+1})_0$ proves $\bigwedge_{m < n} \pi_m \lor \sigma_m$.
\end{proof}

\begin{corollary}
  The theories $\beta_i\Rfn(\Sigma^1_{i+1})_0$ and $\beta_i\Rfn(\B(\Pi^1_{i+1}))_0$ prove the same sentences.
\end{corollary}
\begin{proof}
  It is enough to show that each instance of $\beta_i\Rfn(\B(\Pi^1_{i+1}))$ is a $\B(\Pi^1_{i+1})$ sentence provable from
  $\Sigma^1_{i+1}\hyp\SDC_0$.
  This is immediate from Theorem \ref{thm equiv of SDC and beta}.
\end{proof}

\begin{corollary}
  For any $i > 0$, $\Th_{\B(\Pi^1_{i+1})}(\Sigma^1_i\hyp\SDC_0)$ is not finitely axiomatizable.
\end{corollary}
\begin{proof}
  We show that $\beta_i\Rfn(\B(\Pi^1_{i+1}))_0$ is not finitely axiomatizable.
  For the sake of contradiction, suppose that $\beta_i\Rfn(\B(\Pi^1_{i+1}))_0$ is finitely axiomatizable.
  Then, there are finitely many instances $\rho_0,\ldots,\rho_{n-1}$ of $\beta_i\Rfn(\B(\Pi^1_{i+1}))_0$ such that
  $\ACAo + \bigwedge_{i < n}\rho_i$ axiomatizes $\beta_i\Rfn(\B(\Pi^1_{i+1}))_0$. Let $\rho$ be the conjunction of $\rho_0,\ldots,\rho_n$.
  Since $\rho$ is in $\B(\Pi^1_{i+1})$, we have
  $\ACAo + \rho \vdash \exists \M(\beta_i(\M) \land \M \models \rho)$.
  Therefore, $\ACAo + \rho$ proves the consistency of itself, a contradiction.
\end{proof}

\begin{theorem}
$\Sigma^1_{i}\hyp\SDC_0$ is $\Sigma^1_{i+1}$-conservative over $\ACAo + \{\beta^i_i\RFN^-(n) : n \in \omega\}$.
  In particular,
  $\Pi^1_1\hyp\CAo$ is $\Sigma^1_2$-conservative over the theory $\ACAo + \{\exists Y(Y = \HJ^n(\varnothing)) : n \in \omega\}$.
\end{theorem}
\begin{proof}
  The same argument for Lemma \ref{Lem_cons_or} works.
\end{proof}

In \cite{suzuki_yokoyama_pi12}, it is proved that $\{\exists Y(Y = \HJ^n(\varnothing)) : n \in \omega\}$ is equivalent to $\{(\Sigma^{0,\varnothing}_1)_n\hyp\mathsf{Det} : n \in \omega\}$ over $\ACAo$.
Here, $(\Sigma^{0,\varnothing}_1)_n\hyp\mathsf{Det}$ is the determinacy of Gale-Stewart games with payoff sets of complexity  $(\Sigma^{0,\varnothing}_1)_n$.
Thus, we have the following characterization of $\Sigma^1_2$-consequences of $\Pi^1_1\hyp\CAo$.
\begin{corollary}
  $\Pi^1_1\hyp\CAo$ is $\Sigma^1_2$-conservative over $\ACAo + \{(\Sigma^{0,\varnothing}_1)_n\hyp\mathsf{Det} : n \in \omega\}$. 
\end{corollary}

We then see the relationship between $\beta^i_e\RFN^-$ and $\RFN^i_e$.
\begin{definition}
  Let $\Gamma$ be a class of formulas.
  For theories $T$ and $T'$, we say
  $T$ is a $\Gamma$-extension of $T'$ (written $T' \subseteq_{\Gamma} T)$ if any sentence in $\Gamma$ which is provable from
  $T'$ is also provable from $T$. We write $T' \subsetneq_{\Gamma} T$ if there is a sentence in $\Gamma$ which is provable from
  $T$ but not provable from $T'$ in addition.
  We also write $T =_{\Gamma} T'$ when $T$ and $T'$ are $\Gamma$-equivalent, that is, they prove the same $\Gamma$-sentence.
\end{definition}

As $\ACAo + \{\beta^1_0\RFN(n) : n \in \omega\} =_{\Pi^1_2} \Pi^1_1\hyp\CAo$ and
$\ACAo + \{\exists Y(Y = \HJ^n(\varnothing)) : n \in \omega\} =_{\Sigma^1_2} \Pi^1_1\hyp\CAo$,
these three theories are $\Pi^1_1$-equivalent.
Thus, we have two approximations of the $\Pi^1_1$-consequences of $\Pi^1_1\hyp\CAo$; one is the hierarchy of
$\beta^1_0\RFN(0),\beta^1_0\RFN(1),\beta^1_0\RFN(2),\ldots,$ and the other is the hierarchy of the iterated hyperjumps
of the empty set.
In what follows, we compare these two approximations.

\begin{lemma}\label{lem: separation by beta-models 1}
  Let $T$ and $T'$ be finitely axiomatized theories including $\ACAo$
  such that $T'$ proves the $\beta_i$-model reflection of $T$.
  Then, $T \subsetneq_{\Pi^1_{i+2}} T'$ and $T'$ proves the consistency of $T$.
\end{lemma}
\begin{proof}
  Let $\sigma \equiv \forall X \exists Y \theta(X,Y)$ be a $\Pi^1_{i+2}$ sentence
  provable from $T$.
  We work in $T'$. Take an arbitrary $X$ and a $\beta_i$-model $\M$ such that
  $X \in \M$ and $\M \models T$.
  Then, $\M \models \exists Y \theta(X,Y)$.
  Since $\theta$ is $\Pi^1_i$ and $\M$ is a $\beta_i$ model,
  $\exists Y \theta(X,Y)$ holds.
\end{proof}

By a similar proof, we also have the following.

\begin{lemma}\label{lem: separation by beta-models 2}
  Let $T$ and $T'$ be finitely axiomatized theories including $\ACAo$
  such that $T'$ proves the existence of a $\beta_i$-model of $T$.
  Then, $T \subsetneq_{\Sigma^1_{i+1}} T'$ and $T'$ proves the consistency of $T$.
\end{lemma}

\begin{theorem}
  Let $i,e \in \omega$ such that $e < i $.
  Then, for each $n \in \omega$,
  $\beta^i_e\RFN(n+1)$ proves the $\beta_e$-model reflection of $\beta^i_i\RFN^-(n)$ over $\ACAo$.
  In particular, we have
  $\beta^i_i\RFN^-(n) \subsetneq_{\Pi^1_{e+2}} \beta^i_{e}\RFN(n+1)$ over $\ACAo$.
\end{theorem}

\begin{proof}
  Take an arbitrary $X$. By $\beta^i_e\RFN(n+1)$, take $\M_0,\ldots,\M_{n+1}$ such that
  $X \in \M_{0}$ and $\beta^i_e(\M_0,\ldots,\M_{n+1})$. Then, $\M_{n+1}$ is a $\beta_e$-model and
  $\M_{n+1} \models \beta^i_i(\M_0,\ldots,\M_n)$. In particular, $\M_{n+1} \models \beta^i_i\RFN^{-}(n)$.
  The inculusion $\beta^i_i\RFN^-(n) \subsetneq_{\Pi^1_{e+2}} \beta^i_{e}\RFN(n+1)$ follows from Lemma \ref{lem: separation by beta-models 1}.
\end{proof}

\begin{theorem}\label{thm: Sigma^1_i-1 extension}
  Let $i \in \omega$ such that $i > 0$.
  Then, for each $n \in \omega$, 
  $\beta^i_i\RFN^{-}(n+1)$ proves the existence of a $\beta_i$-model of $\beta^i_{i-1}\RFN(n+1)$.
  In particular, we have 
  $\beta^i_{i-1}\RFN(n+1) \subsetneq_{\Sigma^1_{i+1}} \beta^i_{i}\RFN^-(n+1)$ over $\Sigma^1_{i-1}\hyp\SDC_0$.
\end{theorem}
\begin{proof}
  Let $\M_0,\ldots,\M_{n+1}$ be such that $\beta^i_{i}(\M_0,\ldots,\M_{n+1})$.
  Then, $\M_0$ is a $\beta_i$-model, and $\M_{n+1}$ is a model of $\Sigma^1_{i-1}\hyp\SDC_0$.
  We prove that $\M_0$ is a model of $\beta^i_{i-1}\RFN(n+1)$.
  Since $\M_{n+1}$ is a model of $\Sigma^1_{i-1}\hyp\SDC$, there exists an $\M' \in \M_{n+1}$ such that
  $\M_n \in \M'$ and $\M_{n+1} \models \beta_{i-1}(\M')$.
  We show that $\M'$ believes that $\M_n$ is a $\beta_i$-model.
  Take a $\Pi^1_i$ sentence $\sigma$ with parameters from $\M_n$ such that $\M_n \models \sigma$.
  Then, $\M_{n+1} \models \sigma$ because $\M_n$ is a $\beta_i$-submodel of $\M_{n+1}$.
  Since $\M'$ is a $\beta_{i-1}$-submodel of $\M_{n+1}$, $\M' \models \sigma$.

  Now, for any $X \in \M_0$,  $\M_{n+1}$ satisfies
  \begin{align*}
    (\ast) \,\, \exists Y_0,\ldots,Y_{n+1} (X \in Y_0 \land \beta^i_{i-1}(Y_0,\ldots,Y_n,Y_{n+1}))
  \end{align*}
  via taking $Y_m = \M_m$ for $m < n+1$ and $Y_{n+1} = \M'$.
  Since $(\ast)$ is a $\Sigma^1_{i}$ condition and $\M_0$ is a $\beta_i$-submodel of $\M_{n+1}$,
  $\M_0$ also satisfies $(\ast)$. 
  The inclusion $\beta^i_{i-1}\RFN(n+1) \subsetneq_{\Sigma^1_{i+1}} \beta^i_{i}\RFN^-(n+1)$ follows from Lemma \ref{lem: separation by beta-models 2}.
\end{proof}

\begin{corollary}
  Over $\ACA_0^+$,
  \begin{align*}
    \exists Y(Y= \HJ^{n+1}(\varnothing)) \subsetneq_{\Pi^1_{2}} \beta^1_{0}\RFN(n+1) \subsetneq_{\Sigma^1_{2}} \exists Y(Y = \HJ^{n+2}(\varnothing)).
  \end{align*}
\end{corollary}
\begin{proof}
  It is immediate from the above theorems and the equivalence of the existence of a $\beta$-model and a hyperjump.
\end{proof}

\begin{remark}
  We note that the inclusions in the previous corollary also hold over a theory $T$ such that $T$ includes $\ACA_0^+$ and is a proper subtheory of $\beta^1_0\RFN(n+1)$.
  For example, the inclusions also hold over $\ATRo$.
\end{remark}

\section{Boolean combinations of $\Pi^1_e$ formulas provable from $ \Sigma^1_i\hyp\SDC_0$}
In this section, we generalize the result in the previous sections to give a characterization of $\Th_{\B(\Pi^1_e)}(\Sigma^1_i\hyp\SDC)$ for $e,i \in \omega$ such that
$e < i$.
Most arguments of this section are almost the same as those in the previous section.
That is, we use statements of the form $\sigma \to \beta^i_e\RFN(\sigma;n)$ for the characterization.
The difference between this section and the previous section is that
\begin{itemize}
  \item  $\{\sigma \to \beta^i_e\RFN^-(\sigma;1) : \sigma \in \Sigma^1_{e+1}\}$ does not include $\{\sigma \to \beta^i_e\RFN^-(\sigma;n) :n \in \omega, \sigma \in \Sigma^1_{e+1}\}$,
  \item by our compactness argument, the constructed model may not be a model of $\Sigma^1_i\hyp\SDC_0$.
\end{itemize}
To adresses the first problem, we adopt $\{\sigma \to \beta^i_e\RFN^-(\sigma;n) : n \in \omega\}$ for the characterization.
To adresses the second problem, we slightly modify the proof in the previous section.

We give a comment concering the first item.
\begin{remark}
Let $i,e \in \omega$ such that  $e < i$.
By the same proof of Theorem \ref{thm: Sigma^1_i-1 extension}, the following is provable in
 $\ACA_0$. If $\M_0,\ldots,\M_{n+1}$ are coded $\omega$-models such that
$\beta^i_e(\M_0,\ldots,\M_n)$, then $\M_0$ is a model of
$\{\sigma \to \beta^i_e\RFN^-(\sigma;n) : \sigma \in \Sigma^1_{e+1}\}$.
Therefore, $\beta^i_e\RFN^-(n+1)$ implies the consistency of $\{\sigma \to \beta^i_e\RFN^-(\sigma;n) : \sigma \in \Sigma^1_{e+1}\}$.
Hence, $\{\sigma \to \beta^i_e\RFN^-(\sigma;n) :n \in \omega, \sigma \in \Sigma^1_{e+1}\}$ implies
the consistency of $\{\sigma \to \beta^i_e\RFN^-(\sigma;1) : \sigma \in \Sigma^1_{e+1}\}$.
\end{remark}

\begin{theorem}
  Let $i,e \in \omega, e < i$.
  The following theories are equivalent.
  \begin{enumerate}
    \item $\ACAo + \{\sigma \to \beta^i_e\RFN^-(\sigma;n) : n \in \omega, \sigma \text{ is a $\Sigma^1_{e+1}$ sentence}\}$.
    \item $\Th_{\B(\Pi^1_{e+1})}(\Sigma^1_i\hyp\SDC_0)$.
  \end{enumerate}
\end{theorem}
\begin{proof}
  It is trivial that any instance of (1) is in (2). Therefore, (1) is a subtheory of (2).
  We show the converse inclusion.

 Let $T$ be the theory of (1) 
  As in the previous section, it is enough to show that
  for any $\sigma \in \Sigma^1_{e+1}$ and $\pi \in \Pi^1_{e+1}$,
  if $\Sigma^1_{i}\hyp\SDC_0$ proves $\sigma \lor \pi$ then so does $T$.

  Let $\sigma \in \Sigma^1_{e+1}$ and $\pi \in \Pi^1_{e+1}$ such that
  $\Sigma^1_{i}\hyp\SDC_0$ proves $\sigma \lor \pi$. Assume that $T$ does not prove $\sigma \lor \pi$ for the sake of contradiction.
  Let $\{C_n\}_{n \in \omega}$ be new set constants. Then, the theory
  \begin{align*}
    T' = T+ \lnot \sigma  +  \{\beta^i_e(C_0,\ldots,C_n) : n \in \omega\} + C_0 \models \lnot \tau
  \end{align*}
  is consistent. Let $\M = (\N^{\M},S^{\M},\{A_n\}_{n \in \omega})$ be a model of $T'$.
  Put $S' = \bigcup_{n \in \omega} \{X \in S^{\M} : \M \models X \in A_n\}$ and
  $\M' = (\N^{\M},S')$.
  We note that
  \begin{itemize}
    \item $\M' \models \beta_e(A_n)$ for all $n \in \omega$,
    \item $\M'$ is a $\beta_e$-submodel of $\M$
  \end{itemize}
  by Lemma \ref{lem : closed under beta_i}.
  We show that $\M'$ satisfies both $\sigma \lor \pi$ and $\lnot \sigma \land \lnot \pi$, a contradiction.

  \setcounter{claimcounter}{1} 
  \begin{claim}
    $\M' \models \sigma \lor \pi$.
  \end{claim}
  \addtocounter{claimcounter}{1}
  \begin{proof}[Proof of the claim]
    Since $\sigma \lor \pi$ is a $\Pi^1_{e+2}$ sentence, there exists a $\Pi^1_{e}$ formula $\rho$ such that
    $\sigma \lor \pi \leftrightarrow \forall X \exists Y \rho$.
    Since $e < i$, we can apply Lemma \ref{lemma, number of beta}. Thus, we have $n \in \omega$ such that
    \begin{align*}
      \ACAo \vdash \forall X,\M_0,\ldots,\M_n(\beta^i_i(\M_0,\ldots,\M_n) \to \exists Y \in \M_n \rho(X,Y)).
    \end{align*}
    To prove $\M' \models \forall X \exists Y \rho$, take an arbitrary $X \in \M'$.
    Then, for sufficiently large $k$, we have $X \in A_k$.
    Within $\M'$, the following properties holds.
    \begin{itemize}
      \item $\beta_e(A_{k+n+1})$ and  $A_{k+n+1} \models \ACAo$,
      \item $A_{k+e+1}$ satisfies $X \in A_k \land \beta^i_i(A_k,\ldots,A_{k+n+1})$.
    \end{itemize}
    Therefore,
      $\M' \models \left[ A_{k + n +1} \models \exists Y \rho(X,Y) \right]$.
      Since $\M' \models \beta_e(A_{k+n+1})$ and $\rho$ is a $\Pi^1_e$ formula,
      $\M' \models \exists Y \rho(X,Y)$.
      Consequently, $\M' \models \forall X \exists Y \rho(X,Y)$ and hence
      $\M' \models \sigma \lor \tau$.
  \end{proof}
  \begin{claim}
    $\M' \models \lnot \sigma \land \lnot \pi$.
  \end{claim}
  \addtocounter{claimcounter}{1}
  \begin{proof}[Proof of the claim]
    Since $\lnot \sigma$ is a $\Pi^1_{e+1}$ sentence true in $\M$ and $\M'$ is a $\beta_e$-submodel of $\M$, $\M'$ also satisfies $\lnot \sigma$.
    Since $\lnot \pi$ is a $\Sigma^1_{e+1}$ sentence true in $A_0$ and $A_0$ is a $\beta_i$-submodel of $\M'$,
    $\M'$ also satisfies $\lnot \pi$.
  \end{proof}
  The above claims complete the proof.
\end{proof}

By a similar argument, we have the following characterization of $\Sigma^1_{e+1}$-consequences of $\Sigma^1_i\hyp\SDC_0$.

\begin{theorem}
  Let $i, e \in \omega$, $e < i$. Then, $\Sigma^1_i\hyp\SDC_0$ is $\Sigma^1_{e+1}$ conservative over 
  $\ACAo + \{\beta^i_e\RFN^-(n) : n \in \omega\}$.
\end{theorem}

From the $\Pi^1_4$ equivalence of $\Sigma^1_i\SDC_0$ and $\Pi^1_i\hyp\CAo$, we also have the following.
\begin{corollary}
  Let $i, e \in \omega$, $e < i$ and $e < 2$. Then, 
  $\Pi^1_e\hyp\CAo$ is $\Sigma^1_{e+1}$ conservative over $\ACAo + \{\beta^i_e\RFN^-(n) : n \in \omega\}$.
\end{corollary}

\bibliographystyle{plain}
\bibliography{references}

\begin{thebibliography}{1}

\bibitem{Beklemishev_local_reflection}
Lev Beklemishev.
\newblock Notes on local reflection principles.
\newblock volume~63, pages 139--146. 1997.
\newblock The arithmetization of metamathematics.

\bibitem{pacheco2022determinacy}
Leonardo Pacheco and Keita Yokoyama.
\newblock Determinacy and reflection principles in second-order arithmetic.
\newblock {\em arXiv preprint arXiv:2209.04082}, 2022.

\bibitem{Simpson}
Stephen~G. Simpson.
\newblock {\em Subsystems of second order arithmetic}.
\newblock Perspectives in Logic. Cambridge University Press, Cambridge;
  Association for Symbolic Logic, Poughkeepsie, NY, second edition, 2009.

\bibitem{suzuki_yokoyama_pi12}
Yudai Suzuki and Keita Yokoyama.
\newblock On the {$\Pi^1_2$} consequences of
  {$\Pi^1_1\mathchar`-\mathsf{CA}_0$}.
\newblock {\em arXiv preprint arXiv:2402.07136}, 2024.

\bibitem{Townser_TLPP}
Henry Towsner.
\newblock Partial impredicativity in reverse mathematics.
\newblock {\em J. Symbolic Logic}, 78(2):459--488, 2013.

\end{thebibliography}

\end{document}